\documentclass[11pt,a4paper]{amsart}

\usepackage{graphicx} % Required for inserting images
\usepackage{amssymb}
\usepackage{amsthm}
\usepackage{amsmath}
\usepackage{parskip}
\usepackage[margin=1.1in]{geometry}
\usepackage{keytheorems}
\usepackage{amsfonts,amsbsy,amsgen,amscd,amssymb,amscd}
\usepackage[dvipsnames]{xcolor}
\usepackage{graphicx}
\usepackage{float}
\usepackage{csquotes}
\usepackage{times}
\usepackage[T1]{fontenc}

\usepackage[
  backend=bibtex,
  style=alphabetic,
  doi=true,
  maxbibnames=99, minbibnames=99
]{biblatex}

\usepackage{hyperref}

\newcommand{\kh}{Khovanski{\u\i}}
\newcommand{\khs}{Khovanski{\u\i}'s}
\newcommand{\bzt}{B\'ezout-type}

\usepackage{tikz,tikz-3dplot,tikz-cd}
\usetikzlibrary{matrix,arrows,shapes,calc,3d}
\usepackage{pgfplots}

\pgfplotsset{compat=newest}
\definecolor{dark-gray}{gray}{0.3}

\newkeytheorem{example}[parent=section]
\newkeytheorem{definition}[sibling=example]
\newkeytheorem{remark}[sibling=example]
\newkeytheorem{theorem}[sibling=example]
\newkeytheorem{proposition}[sibling=example]
\newkeytheorem{lemma}[sibling=example]
\newkeytheorem{corollary}[sibling=example]
\newkeytheorem{claim}[sibling=example]
\numberwithin{equation}{section}

\newcommand{\cU}{\mathcal{U}}
\newcommand{\boldq}{\mathbf{q}}

\newcommand{\eps}{\varepsilon}

\DeclareMathOperator{\pool}{\mathcal{A}}

\newcommand{\trans}{\intercal}
\newcommand{\econst}{\mathrm{e}}

\providecommand{\mathbbm}{\mathbb} % In case we don't load bbm
\newcommand{\R}{\mathbbm{R}}

\newcommand{\C}{\mathbbm{C}}
\newcommand{\N}{\mathbbm{N}}

\newcommand{\abs}[1]{\left\vert {#1} \right\vert}
\newcommand{\diff}[1]{\mathrm{d}{#1}}
\newcommand{\de}{\mathrm{d}}
\newcommand{\suchthat}{\ \text{ such that }\ }
\newcommand{\ip}[2]{\langle {#1}, {#2} \rangle}
\newcommand{\norm}[1]{\Vert {#1}\Vert}
\newcommand{\Pfaff}{\mathrm{Pfaff}}

\title{Khovanski{\u\i}'s B\'ezout-type Theorem for Pfaffian Functions: A Self-Contained Proof, and Applications}

\author{Martin Lotz}
\author{Abhiram Natarajan}
\date{}

\begin{document}

\begin{abstract}
We present a direct and self-contained proof of \khs{} \bzt{} bound for the number of nondegenerate solutions of a system of Pfaffian equations. We isolate the ingredients of \khs{} original argument and assemble them into a proof that avoids the general theory of integral manifolds developed in the monograph \cite{khovanskiui1991fewnomials}. Our formulation mildly refines the classical statement~\cite[\S 3.12, Corollary~5]{khovanskiui1991fewnomials}: rather than depending on the ambient dimension, our bound depends on the maximum number of variables on which any function in the Pfaffian chain depends. As a consequence, we obtain a refined bound on the number of connected components of a Pfaffian set.
\end{abstract}

\maketitle

\section{Introduction}
A fundamental result in the theory of Pfaffian functions is the \bzt{} theorem of \kh{} \cite{khovanskii1980class,khovanskiiicm,khovanskiui1991fewnomials}, which provides an explicit bound on the number of non-degenerate solutions of a system of Pfaffian equations in terms of the combinatorial data of the functions involved. This result plays a central role in the study of Pfaffian sets and has had wide-ranging influence: it underpins finiteness results in real analytic geometry and the theory of o-minimal structures~\cite{wilkie1996model,wilkie1999theorem,speissegger1999pfaffian,van1998tame}, and it lies at the origin of \khs{} theory of fewnomials and the ensuing work on sharp bounds for the number of real solutions of sparse polynomial systems~\cite{khovanskiui1991fewnomials,li2003counting,bihan2007fewnomial,sottile2011real}. It also serves as a key ingredient in recent work on polynomial partitioning, incidence geometry, bounds for tubular neighbourhoods, and neural networks \cite{bianchini2014complexity,lotz2024partitioning,natarajan2025,lezeaulotz2026}. 

In~\cite[\S3.12]{khovanskiui1991fewnomials}, Theorem~\ref{thm:khovanskii} is derived as a corollary of more general estimates on integral manifolds of differential forms with polynomial coefficients. Our main aim is to isolate the ingredients of \khs{} argument and assemble them into a direct proof that can be read independently of the surrounding theory. The resulting proof follows \khs{} original strategy while avoiding the broader machinery developed in the monograph.

\begin{theorem}[note={\kh{} \cite{khovanskii1980class,khovanskiiicm,khovanskiui1991fewnomials}},store=thm-khovanskii]
\label{thm:khovanskii}
Let $f_1,\dots,f_n$ be Pfaffian functions, defined on all of $\R^n$, with chain $\boldq=(q_1,\dots,q_s)$ and component-wise formats $(\alpha, \beta_i, s)$, such that the functions in the Pfaffian chain depend only on $\{x_1,\dots,x_k\}$ for some $k\leq n$. Then the number of regular real solutions of the system $f_1(x) = \cdots = f_n(x)= 0$ is bounded by 
\begin{equation*}
2^{\frac{s(s - 1)}{2}} \beta_1 \dotsm \beta_n \left(\beta_1 + \dotsb + \beta_n - n + \min\{k, s\}\alpha + 1 \right)^s.
\end{equation*}
\end{theorem}

Theorem~\ref{thm:khovanskii} records the number $k$ of variables on which the functions in the Pfaffian chain depend, whereas in the classical formulation~\cite[\S 3.12, Corollary~5]{khovanskiui1991fewnomials}, the last factor features $\min\{n, s\}\alpha$ in place of $\min\{k, s\}\alpha$. This formulation is needed to close the induction in our proof. Moreover, it immediately yields corresponding improvements to bounds on the Betti numbers of Pfaffian, semi-Pfaffian, and sub-Pfaffian sets. As an illustration, we obtain the following bound on the number of connected components of a Pfaffian set (proved in Section~\ref{sec:applications}).

\begin{theorem}[note={Bound on the zeroth Betti number of a Pfaffian set}, store=thm-pfaffian-connected-components] \label{thm:pfaffian-connected-components}
   Let $f_1, \ldots, f_m$ be Pfaffian functions defined on all of $\R^n$, with chain $\boldq=(q_1,\dots,q_s)$, each of format $(\alpha, \beta, s)$, for which the functions in the Pfaffian chain depend only on $k \leq n$ of the variables. Then the number of connected components of the set 
   \[
   \mathcal{Z}(f_1, \ldots, f_m) := \{x\in \R^n\colon g_1(x)=\cdots = g_k(x)=0\}
   \]
   is bounded by
 \begin{equation}\label{eq:jones}
 2^{\frac{s(s - 1)}{2}+1} (\alpha+2\beta-1)^{n-1}\beta \left( n (\alpha+2\beta-2)+\alpha(\min\{k, s\}-1) + 2 \right)^s.
 \end{equation}
 \end{theorem}

As mentioned, the above theorem, which is a variant of~\cite[Corollary 3.3]{gabrielov2004complexity} and~\cite[Theorem 2.3]{jones2012density}, is proved using the sharpened formulation of Theorem~\ref{thm:khovanskii}. The refinement is advantageous when the Pfaffian chain is short ($s = o(n)$) or depends on few of the variables ($k = o(n)$). While the bound~\eqref{eq:jones} is not as sharp as possible, it gives a convenient way of counting the number of solutions to a system of Pfaffian equations without reference to the number of equations. We note that Theorem \ref{thm:pfaffian-connected-components} is only meant to be illustrative of how Theorem \ref{thm:khovanskii} can be used; similar refinements can be obtained for many existing bounds on Betti numbers of sets defined by Pfaffian functions such as in \cite{zell1999betti, zellthesis, gabrielov2004semisubpfaffian}.

\subsection{Related work}
An expository account that establishes the finiteness of the set of solutions is given by Marker~\cite{Marker1997}. 
Wilkie~\cite{wilkie1996model} gave a model-theoretic proof, without an explicit bound, of \khs{} theorem that the zero set of a Pfaffian function has finitely many connected components, uniformly over parameters \cite[Theorem 5.3]{wilkie1996model}.
 Bounds on the number of connected components of a Pfaffian set can be derived from \khs{} theorem by adapting an argument from semi-algebraic geometry~\cite[Chapter 11]{bochnak2013real}; see~\cite[Corollary 3.3]{gabrielov2004complexity} and~\cite[Theorem 2.3]{jones2012density} for a statement of such a bound.
  This bound serves as the starting point for a broader quantitative theory of the topology of sets defined by Pfaffian functions. For semi-Pfaffian sets, Zell~\cite{zell1999betti,zell2003quantitative} established explicit bounds, in terms of the format, on the sum of all Betti numbers, going beyond the count of connected components. Gabrielov, Vorobjov, and Zell~\cite{gabrielov2004semisubpfaffian} extended such bounds to sub-Pfaffian sets. The sharpness of \khs{} bound is studied in~\cite{bickerton2026}.

\subsection*{Acknowledgements} We are grateful to Alex Wilkie, Nicolai Vorobjov, and Adam Sheffer for helpful discussions that highlighted the value of a direct and accessible exposition of Theorem \ref{thm:khovanskii}. AN was supported by EPSRC Grant EP/V003542/1.

\section{Preliminaries}
\label{sec:pfaffian}
References for the material in this section are~\cite{khovanskiui1991fewnomials, zell2003quantitative, gabrielov2004complexity}. We assume throughout that $n\in \N_{>0}$, and for $m \in \N_{>0}$ we write $[m] := \{1, \ldots, m\}$. For $d,n\in \N$, we denote by $\R[X_1, \ldots, X_n]_{\le d}$ the vector subspace of the polynomial ring $\R[X_1,\dots,X_n]$ containing polynomials of degree at most $d$. We often implicitly identify a polynomial $P$ with the function $P\colon \R^n\to \R$ given by $P(x)=P(x_1,\dots,x_n)$ for $x\in \R^n$. We call a solution $a\in \R^n$ of a system of equations $F(x)=0$ regular (or non-degenerate), if the differential $\diff{F}(a)$ has maximal rank. 

\subsection{Pfaffian functions}
Pfaffian functions, introduced by \kh{}~\cite{khovanskiui1991fewnomials}, are functions that satisfy triangular systems of first-order partial differential equations with polynomial coefficients. Pfaffian functions have a well-defined notion of complexity that can be used to prove various quantitative results about the geometric objects they define.

\begin{definition}[Pfaffian function]\label{def:pfaffian}
Let $\cU \subset \R^n$ be an open set. A \emph{Pfaffian chain} of order $s \in \N$ and chain-degree $\alpha \in \N_{>0}$ over $\cU$ is a sequence of functions $\boldq= (q_1, \dotsc, q_s)$ with $q_i$ real-analytic on $\cU$ for $i\in [s]$, such that there exist polynomials $P_{i,j}\in \R[X_1,\dots,X_n,Y_1,\dots,Y_i]_{\leq \alpha}$, for $i\in [s]$ and $j\in [n]$, that verify
\begin{equation}\label{eq:pfaffchain}
\frac{\partial q_i}{\partial x_j}(x) = P_{i,j}(x, q_1(x), \dotsc, q_i(x)).
\end{equation}
A function $g(x)=P(x,q_1(x),\dots,q_s(x))$, with $P\in \R[X_1,\dots,X_n,Y_1,\dots,Y_s]_{\leq \beta}$ for $\beta\in \N$, is called a \emph{Pfaffian function} of chain degree $\alpha$, degree $\beta$, and order $s$.  A function $\cU\to \R^m$ is called Pfaffian if all its components are Pfaffian. 
\end{definition}

The triple $(\alpha, \beta, s)$ is called a \emph{format} of $g$. We denote by $\Pfaff_{\boldq,\beta}(\cU)$ the set of all Pfaffian functions over $\cU$ with chain $\boldq$ and degree $\beta$, and by $\Pfaff_{\alpha,\beta,s}(\cU)$ the set of all Pfaffian functions over $\cU$ with format $(\alpha,\beta,s)$.  When we say that a Pfaffian function $\cU\to \R^m$ has a particular format, we mean that every component has that same format, and we write $\Pfaff_{\alpha,\beta,s}(\cU; \R^m)$.

\begin{remark}[Uniqueness]
 Note that the Pfaffian chain associated with a Pfaffian function, and hence also a format, is not unique. In particular, if $g\in \Pfaff_{\alpha,\beta,s}(\cU)$, then $g\in \Pfaff_{\alpha',\beta',s'}(\cU)$ for $\alpha'\geq \alpha$, $\beta'\geq \beta$, and $s'\geq s$. In the following, we will often leave the chain implicit, and our results will be stated in terms of the format.
\end{remark}

\begin{remark}[Domain of definition]\label{rem:global-domain}
  The definition of a Pfaffian function involves its domain $\cU$, but 
  our proof of Theorem~\ref{thm:khovanskii} is restricted to $\cU=\R^n$.
  The bound itself, though, is not special to $\R^n$: it remains valid on the domains usually considered. Establishing it there requires the \kh{}--Rolle step of Section~\ref{sec:kho-rolle} to be carried out on non-compact curves.
  \end{remark}

\begin{remark}[Real-analyticity]\label{rem:analytic}
Since the chain $\boldq$ is real-analytic and the representing polynomial is analytic, every Pfaffian function $g = P(x, q_1, \dots, q_s)$ is real-analytic. Our arguments below, however, use only the smoothness of Pfaffian functions; in particular, the \kh{}--Rolle step (Lemma~\ref{lem:kho-rolle} and Corollary~\ref{cor:kho-rolle}) is stated for smooth maps. Requiring the chain to be merely $C^1$ would impose no real loss of generality: the analytic system~\eqref{eq:pfaffchain} admits an analytic solution, which by uniqueness coincides with the given $C^1$ one (see~\cite[Theorem 2.4.1]{krantz2002primer}).
\end{remark}

\begin{example}\label{ex:1} The following are simple examples of Pfaffian functions.
        \begin{enumerate}
            \item A polynomial $P \in \R[X_1, \dots, X_n]$ gives rise to a Pfaffian function with format $(\alpha,\deg P,0)$ for any $\alpha>0$ via the evaluation homomorphism;
            \item $\exp\colon \R\to \R$ is Pfaffian with format $(1,1,1)$. A Pfaffian chain is $\boldq=(\exp)$ and $P_{1,1}(X,Y_1)=Y_1$;
            \item $\tanh$ is Pfaffian with format $(2,1,1)$. A Pfaffian chain is $\boldq=(\tanh)$ and $P_{1,1}(X,Y_1)=1-Y_1^2$;
            \item The logistic sigmoid $\sigma = (1 + \econst^{-x})^{-1}$ is Pfaffian with format $(2, 1, 1)$. A Pfaffian chain is $\boldq=(\sigma)$ and $P_{1,1}(X,Y_1)=Y_1(1-Y_1)$;
            \item $\arctan$ is Pfaffian with format $(3, 1, 2)$. A Pfaffian chain is $\boldq=((1+x^2)^{-1},\arctan(x))$. The polynomials are $P_{1,1}(X,Y_1) = -2XY_1^2$ and $P_{2,1}(X,Y_1,Y_2)=Y_1$. 
            \item \label{point:fewnomials-example} For any monomial $m_{i_1, \ldots, i_n} := a_{i_1, \ldots, i_n}x_1^{i_1}\ldots x_n^{i_n}$ with $a_{i_1, \ldots, i_n} \neq 0$, $\left(\frac{1}{x_1}, \ldots, \frac{1}{x_n}, m_{i_1, \ldots, i_n}\right)$ is a Pfaffian chain on the domain $\{x \in \R^n \suchthat x_1\dots x_n \neq 0\}$ of order $n+1$, and chain-degree  $2$, given that
\begin{equation*}
\de m_{i_1, \ldots, i_n} = m_{i_1, \ldots, i_n}\cdot \sum_{j=1}^n i_j \frac{1}{x_j} \de x_j, \qquad \text{and} \qquad \de \frac{1}{x_j} = -\left(\frac{1}{x_j}\right)^2 \de x_j.
\end{equation*}
From this we can deduce that if $f$ is a polynomial that is a sum of $s$ such monomials, called a \emph{fewnomial of sparsity $s$}, then $f$ is a Pfaffian function (on the same domain as before) with format $(2, 1, n+s)$.
        \end{enumerate}
\end{example}

Pfaffian functions enjoy convenient closure properties under algebraic operations and composition, as illustrated in the following two results (see also~\cite[Proposition 1.8]{zell2003quantitative} and~\cite[Section 2]{lezeaulotz2026}).

\begin{lemma}\label{le:pfaffalgebra}
  Let $g\in \Pfaff_{\alpha, \beta, s}(\cU)$ and $h\in \Pfaff_{\alpha', \beta', s'}(\cU)$. Then: 
  \begin{enumerate}
  \item $g + h \in \Pfaff_{\max\{\alpha,\alpha'\}, \max\{\beta,\beta'\},s+s'}(\cU)$;
  \item $gh \in \Pfaff_{\max\{\alpha,\alpha'\}, \beta+\beta',s+s'}(\cU)$;
  \item For each $i\in [n]$, $\partial g/\partial x_i \in \Pfaff_{\alpha,\alpha+\beta-1,s}(\cU)$. 
  \end{enumerate}
  \end{lemma}
  
  \begin{remark}\label{re:1}
  Note that the bound on the order of $g+h$ and $gh$ is in general not sharp: if $g$ and $h$ are Pfaffian with respect to the same chain $\boldq$ of order $s$, then clearly the order of $g+h$ and $gh$ is also $s$. It also follows from Lemma~\ref{le:pfaffalgebra} that a linear combination of Pfaffian functions $g_1,\dots,g_m$ with formats $(\alpha_i,\beta_i,s_i)$ is Pfaffian with format $(\max_i \{\alpha_i\},\max_i\{\beta_i\}, \sum_i s_i)$. 
  \end{remark}
  
  \begin{lemma}\label{le:composition}
  Let $g\in \Pfaff_{\alpha,\beta,s}(\cU;\R^m)$ and $h\in \Pfaff_{\alpha',\beta',s'}(\R^m)$ be Pfaffian functions, and assume $s'\geq 1$. 
  \begin{enumerate}
  \item $h\circ g \in \Pfaff_{(\alpha'+1)\beta+\alpha-1, \beta\beta', ms+s'}(\cU)$;
  \item If all the functions in $g$ depend on the same Pfaffian chain $\boldq$ of order $s$, then $h\circ g$ has Pfaffian order $s+s'$.
  \end{enumerate}
  \end{lemma}

  In particular, in the special case where $g\colon \R^n\to \R^m$ is a linear map shows that the Pfaffian structure is invariant under affine change of coordinates.

The two main ingredients of the proof of Theorem~\ref{thm:khovanskii} are B\'ezout's Inequality for real algebraic sets, Theorem~\ref{thm:real-bezout}, and a generalization of Rolle's Theorem~\cite[\S 3.2]{khovanskiui1991fewnomials}.

\begin{theorem}[Real B\'ezout Inequality] \label{thm:real-bezout}
  Let $P_1, \ldots, P_n \in \R[X_1, \ldots, X_n]$ be polynomials of degrees $d_1, \dots, d_n$. Then the number of non-degenerate solutions of the system
  \begin{equation*}
  P_1(x) = \cdots = P_n(x) = 0
  \end{equation*}
  is bounded by the product of the degrees $d_1 \cdots d_n$.
  \end{theorem}

\begin{proof}
For each $i \in [n]$, let $P_i^h \in \C[X_0, X_1, \ldots, X_n]$ denote the homogenization of $P_i$, a homogeneous polynomial of degree $d_i$ with $P_i^h(1, x) = P_i(x)$. A non-degenerate solution $x \in \R^n$ of the system is a non-singular solution over $\C$, and thus corresponds to the non-singular projective solution $(1 : x_1 : \cdots : x_n)$ of the system $P_1^h = \cdots = P_n^h = 0$ (see~\cite[\S 11.5]{bochnak2013real}); distinct solutions give rise to distinct projective points. By~\cite[Lemma~11.5.1]{bochnak2013real}, the number of non-singular projective solutions of a system of $n$ homogeneous polynomials in $n+1$ variables is bounded by the product of the degrees, and the claim follows.
\end{proof}

\section{\kh{}--Rolle Theory}
\label{sec:kho-rolle}
In the following, we refer to two roots $a$, $b$ of a function $g\colon C\to \R$ on a smooth $1$-manifold $C\subset \R^{n+1}$ as {\em consecutive roots} if there is an arc connecting $a$ and $b$ that does not contain any other root of $g$. A point $x$ is said to be between two consecutive roots $a$ and $b$ if $x$ is on the arc between $a$ and $b$ that does not contain any other roots. A function $f\colon C\to \R$ is said to \emph{change sign} at $x$, if $f(x)=0$ and there exists an open neighbourhood $U\subset C$, $x\in U$, such that $U\backslash \{x\}$ is the union of two connected components, and $f$ has different signs on each of them. 
For an illustration of the following generalization of Rolle's Lemma, see Figure~\ref{fig:kho-rolle}, as well as~\cite{bates2011khovanskii}.

\begin{lemma}[\kh{}--Rolle] \label{lem:kho-rolle}
Let $G = (g_1, \ldots, g_n)\colon \R^{n+1} \rightarrow \R^n$ be a smooth map and let $y \in \R^n$ be a regular value of $G$, so that $C = G^{-1}(y)$ is a smooth $1$-submanifold of $\R^{n+1}$. Let $g\colon \R^{n+1} \rightarrow \R$ be a smooth map such that all the roots of $g|_C$ are non-degenerate, i.e., the differential $\diff(g|_C)$ does not vanish at any root. Then for any two distinct consecutive roots $a$ and $b$ of $g|_C$, the determinant
\begin{equation}\label{eq:kho-rolle-det}
 \det([\nabla g_1, \ldots, \nabla g_n, \nabla g])
\end{equation}
is non-zero at $a$ and at $b$, and takes values of opposite signs at these two points. In particular, it vanishes at some point between $a$ and $b$.
\end{lemma}

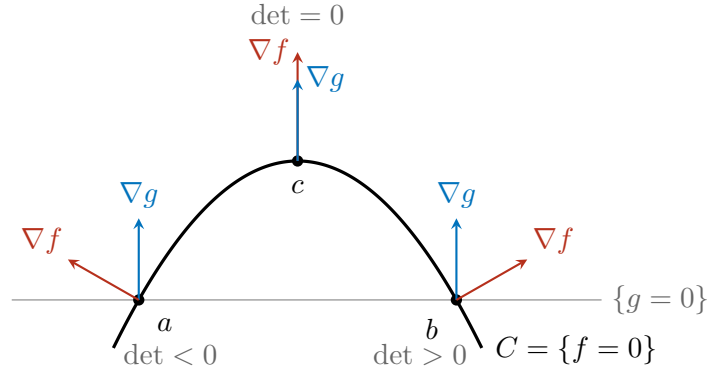
\begin{figure}[ht]
  \centering
  \begin{tikzpicture}[>=stealth, scale=1.2,
      gvec/.style={->, thick, RoyalBlue},
      nvec/.style={->, thick, BrickRed},
      dot/.style={circle, fill, inner sep=1.5pt}]
    % the line {g = 0} (the x-axis, since g(x,y) = y)
    \draw[black!50] (-2.9,0) -- (3.6,0) node[right, black!60] {$\{g = 0\}$};
    % the curve C = {f = 0}:  y = -x^2/2 + x/4 + 3/2
    \draw[very thick] plot[domain=-1.78:2.28, smooth, samples=80]
      (\x, {-\x*\x/2 + \x/4 + 1.5});
    \node[right] at (2.32,-0.55) {$C = \{f = 0\}$};
    % points: a = (-1.5,0), b = (2,0), c = (0.25, 1.53125)
    \node[dot, label={[label distance=2pt]below right:{$a$}}] (a) at (-1.5,0) {};
    \node[dot, label={[label distance=2pt]below left:{$b$}}] (b) at (2,0) {};
    \node[dot, label={[label distance=1pt]below:{$c$}}] (c) at (0.25,1.53125) {};
    % gradients of g (vertical at every point)
    \draw[gvec] (a.center) -- +(0,0.9) node[above] {$\nabla g$};
    \draw[gvec] (b.center) -- +(0,0.9) node[above] {$\nabla g$};
    % outer normals \nabla f
    \draw[nvec] (a.center) -- +(-0.781,0.446) node[above left=-2pt] {$\nabla f$};
    \draw[nvec] (b.center) -- +(0.781,0.446) node[above right=-2pt] {$\nabla f$};
    % at c the two gradients are parallel: drawn slightly offset
    \draw[nvec] ($(c.center)+(-0.00,0)$) -- +(0,1.2) node[left=-1pt] {$\nabla f$};
    \draw[gvec] ($(c.center)+(0.00,0)$)  -- +(0,0.9) node[right=-1pt] {$\nabla g$};
    % sign of the Rolle determinant det([\nabla f, \nabla g])
    \node[black!60] at (-1.15,-0.6) {$\det<0$};
    \node[black!60] at (1.6,-0.6)  {$\det>0$};
    \node[black!60] at (0.25,3.15) {$\det=0$};
  \end{tikzpicture}
  \caption{Lemma~\ref{lem:kho-rolle} in the plane ($n = 1$, $G = f$), with $C = f^{-1}(0)$ and $g(x_1, x_2) = x_2$. At the consecutive roots $a$ and $b$ of $g|_C$, the determinant $\det([\nabla f, \nabla g])$ is non-zero and takes values of opposite signs. It vanishes at the intermediate point $c$, where the normal $\nabla f$ to the curve is parallel to $\nabla g$.}
  \label{fig:kho-rolle}
\end{figure}

\begin{proof}
The $1$-form $\omega(v) = \det([\nabla g_1, \ldots, \nabla g_n, v])$ on $\R^{n+1}$ induces a vector field $X_{\omega}$ that satisfies $\ip{X_{\omega}}{v} = \omega(v)$ for any $v \in T_{x}\R^{n+1}$. Since $y \in \R^n$ is a regular value of $G$, $X_{\omega}$ restricts to a non-vanishing tangent vector field on $C$.
Let $C_0 \subseteq C$ be an arc joining $a$ and $b$ that contains no other root of $g|_C$. Let $I=[t_1,t_2]$ be an interval, and let $\gamma : I \rightarrow C_0$ be a parameterization of $C_0$ such that $\dot{\gamma}(t) = X_{\omega}(\gamma(t))$ for all $t$ (see \cite[Chapter 14]{tu2011manifolds} for a proof of existence of such a parameterization), so that $\{\gamma(t_1), \gamma(t_2)\} = \{a, b\}$. Consider the function $h := g \circ \gamma\colon I \rightarrow \R$, so that $h(t_1) = h(t_2) = 0$ and $h(t)\neq 0$ for all $t\in I\backslash \{t_1, t_2\}$.
The derivative of $h$ at a point $t \in I$ is then given by
\begin{equation*}
  \dot{h}(t) = \ip{X_{\omega}(\gamma(t))}{\nabla g(\gamma(t))} = \omega(\nabla g(\gamma(t))).
\end{equation*}
Note that $\dot{h}(t) = \det([\nabla g_1, \ldots, \nabla g_n, \nabla g])$ evaluated at $\gamma(t)$.

Since $\dot{h}(t_i) = \diff{(g|_C)}\bigl(X_\omega(\gamma(t_i))\bigr)$, where $X_\omega$ does not vanish on $C$ and the roots of $g|_C$ are non-degenerate, we have $\dot{h}(t_1) \neq 0$ and $\dot{h}(t_2) \neq 0$. We first show that $\dot{h}(t_1)$ and $\dot{h}(t_2)$ must have opposite signs. Indeed, if both have the same sign, say they are both positive, then $h$ is increasing at $t_1$ and increasing at $t_2$. Since $h(t_1) = 0$ and $\dot{h}(t_1) > 0$, there exists $\eps_1 > 0$ with $h(t_1 + \eps_1) > 0$. Since $h(t_2) = 0$ and $\dot{h}(t_2) > 0$, there exists $\eps_2 > 0$ with $h(t_2 - \eps_2) < 0$. Choosing $\eps_1, \eps_2 < (t_2-t_1)/2$, the intermediate value theorem gives a point $s\in (t_1+\eps_1, t_2-\eps_2)$ with $h(s) = 0$, contradicting the assumption that $\gamma(t_1)$ and $\gamma(t_2)$ are consecutive roots. The case where both derivatives are negative is analogous.

Since $\dot{h}(t)$ is the determinant~\eqref{eq:kho-rolle-det} evaluated at $\gamma(t)$, it follows that the determinant is non-zero and takes values of opposite signs at $\gamma(t_1)$ and $\gamma(t_2)$, that is, at $a$ and $b$. Being continuous, the determinant then vanishes at some point of $C_0$ between $a$ and $b$ by the intermediate value theorem.
\end{proof}

\begin{corollary} \label{cor:kho-rolle} Let $G = (g_1, \ldots, g_n)\colon \R^{n+1} \rightarrow \R^{n}$ be a smooth map and let $y \in \R^n$ be a regular value of $G$, so that $C = G^{-1}(y)$ is a smooth $1$-submanifold of $\R^{n+1}$. Assume that $C$ is compact. Let $g\colon \R^{n+1} \rightarrow \R$ be a smooth map such that all the roots of $g|_C$ are non-degenerate, and let $\psi\colon \R^{n+1} \rightarrow \R$ be a smooth function that coincides with the determinant~\eqref{eq:kho-rolle-det} at every point of $g^{-1}(0)$. Then there exists a regular value $\delta > 0$ of $\psi|_C$ such that
\begin{equation*}
  \abs{\left\{x \in C \suchthat g(x) = 0\right\}} \le \abs{\left\{x \in C \suchthat \psi(x) = \delta\right\}}.
\end{equation*}
\end{corollary}

\begin{proof}
Since $C$ is a compact $1$-manifold, it has finitely many connected components, each diffeomorphic to $S^1$~\cite[Appendix]{milnor1997topology}. The roots of $g|_C$ are non-degenerate, hence isolated, and therefore finite in number. Moreover, $g|_C$ changes sign at every root, since $\diff(g|_C)$ does not vanish there. If $g|_C$ has no roots, the claimed inequality holds trivially for any regular value $\delta > 0$ of $\psi|_C$, and such a value exists by Sard's theorem. Assume therefore that $g|_C$ has at least one root.

Consider a connected component $\Gamma$ of $C$ containing $p \geq 1$ roots of $g|_C$. The number $p$ is even: choosing a point $c \in \Gamma$ with $g(c) \neq 0$ and traversing $\Gamma$ once, starting and ending at $c$, the sign of $g$ changes exactly at the roots and returns to its initial value, so the number of sign changes, and hence of roots, is even. In particular, $p \geq 2$. Let $\theta\colon S^1 \rightarrow \Gamma$ be a diffeomorphism and let $\theta(t_1), \ldots, \theta(t_p)$ be the roots of $g$ on $\Gamma$, listed in cyclic order. For each $i \in [p]$, the roots $\theta(t_i)$ and $\theta(t_{(i \bmod p)+1})$ are distinct and consecutive, and bound an arc $A_i \subseteq \Gamma$ containing no other roots; the arcs $A_1, \ldots, A_p$ cover $\Gamma$ and have pairwise disjoint interiors. By Lemma~\ref{lem:kho-rolle}, the determinant~\eqref{eq:kho-rolle-det} is non-zero and takes values of opposite signs at the two endpoints of each arc $A_i$. Since these endpoints are roots of $g$, and $\psi$ coincides with the determinant~\eqref{eq:kho-rolle-det} on $\{g = 0\}$, the same holds for $\psi$.

In particular, $\psi$ does not vanish at any root of $g|_C$, so that
\begin{equation*}
  \eps := \min\left\{\abs{\psi(x)} \suchthat x\in C, \ g(x) = 0\right\} > 0,
\end{equation*}
the minimum being over a non-empty finite set. Fix $\delta \in (0, \eps)$ and consider an arc $A_i$ as above. At one endpoint of $A_i$ we have $\psi \geq \eps > \delta$, and at the other $\psi \leq -\eps < \delta$. Since $\psi$ is continuous, the intermediate value theorem yields a point $z_i \in A_i$ with $\psi(z_i) = \delta$. Because $\abs{\psi} \geq \eps > \delta$ at the endpoints, $z_i$ lies in the interior of $A_i$. As the interiors of the arcs are pairwise disjoint, the points $z_1, \ldots, z_p$ are distinct, so $\Gamma$ contains at least $p$ points at which $\psi = \delta$. Summing over the connected components of $C$ shows that, for every $\delta \in (0, \eps)$,
\begin{equation*}
  \abs{\left\{x\in C \suchthat g(x) = 0\right\}} \leq \abs{\left\{x\in C \suchthat \psi(x) = \delta\right\}}.
\end{equation*}
By Sard's theorem, we can choose $\delta \in (0, \eps)$ to be a regular value of $\psi|_C$.
\end{proof}

The \kh{}--Rolle Lemma, Lemma~\ref{lem:kho-rolle}, and its consequence, Corollary~\ref{cor:kho-rolle}, are general results that do not 
require our functions to be Pfaffian. The Pfaffian structure enters through the following technical lemma, 
which bounds the degree of the determinant arising in the \kh{}--Rolle step.

\begin{lemma}[Rolle determinant degree]\label{lem:rolle-det-degree}
Let $\boldq = (q_1,\dots,q_{s})$ be a Pfaffian chain in $\R^n$ of degree $\alpha$, with chain polynomials $P_{i,j}$ and where the chain depends on $k\leq n$ of the variables. Suppose that in addition, we are given:
\begin{enumerate}
\item Functions $h_{s,j} := P_{s,j}(x_1,\dots,x_k, q_1(x),\dots,q_{s-1}(x), x_{n+1})$ for $j \in [k]$;
\item Polynomials $Q_1,\dots, Q_n\in \R[X_1,\dots,X_n,Y_1,\dots,Y_s]$ of degree $\deg Q_i\leq \beta_i$ and an associated map
$G = (g_1,\dots,g_n)\colon \R^{n+1}\to \R^n$ defined by
\begin{equation*}
  g_i(x, x_{n+1}) = Q_i(x, q_1(x),\dots,q_{s-1}(x), x_{n+1}), \qquad i\in [n].
\end{equation*}
\end{enumerate}
Let $M$ be the $(n+1)\times (n+1)$ matrix $M = [\nabla g_1, \dots, \nabla g_n, \eta]$, where
\begin{equation*}
  \eta(x, x_{n+1}) = \bigl(h_{s,1},\dots,h_{s,k}, 0,\dots,0, -1\bigr)^{\trans}.
\end{equation*}
Then $\det(M)$ is a Pfaffian function with format $(\alpha, \beta_{n+1}, s-1)$ and chain $(q_1,\dots,q_{s-1})$, where
\begin{equation}\label{eqn:rolle-det-degree}
  \beta_{n+1} \leq \sum_{i=1}^n (\beta_i - 1) + \min\{k, s\}\alpha.
\end{equation}
\end{lemma}

Figure~\ref{fig:rolle-det} illustrates the geometric content of this construction.

\begin{figure}[ht]
  \centering
  \begin{tikzpicture}[>=stealth, scale=1.35,
      leaf/.style={black!25, thin},
      tanleaf/.style={black!45, thin},
      dot/.style={circle, fill, inner sep=1.5pt},
      tdot/.style={diamond, fill=BrickRed, inner sep=1.6pt}]
    \clip (-1.85,-0.2) rectangle (3.45,3.95);
    % foliation: leaves x2 = c e^{x1} of the chain equation
    \foreach \c/\dmax in {0.05/2.05, 0.12/2.05, 0.55/1.86, 1.8/0.67, 6.8/-0.66}
      \draw[leaf] plot[domain=-1.65:\dmax, samples=50] (\x, {\c*exp(\x)});
    % the two leaves tangent to C (slightly darker)
    \draw[tanleaf] plot[domain=-1.65:2.05, samples=50] (\x, {0.3115*exp(\x)});
    \draw[tanleaf] plot[domain=-1.65:-0.09, samples=50] (\x, {3.8898*exp(\x)});
    % N = the leaf c = 1 (graph of q_s)
    \draw[RoyalBlue, very thick] plot[domain=-1.65:1.26, samples=60] (\x, {exp(\x)});
    \node[RoyalBlue, right] at (1.30,3.42) {$N$};
    \node[black!55, right] at (-1.8,3.72) {\small leaves $x_2 = c\,\econst^{x_1}$};
    % the algebraic tangency locus {det M = 0}
    \draw[dashed, thick, BrickRed] plot[domain=0.2:2.7, samples=60] ({0.5-\x*(\x-1.9)}, \x);
    \node[BrickRed, right] at (1.0,0.28) {$\{\det(M) = 0\}$};
    % the lifted curve C = {g_1 = 0}
    \draw[very thick] (0.5,1.9) circle (1.1);
    \node at (2.5,1.62) {$C = \{g_1 = 0\}$};
    % roots of g|_C (intersections with N)
    \node[dot] (r1) at (-0.0525,0.9488) {};
    \node[dot] (r2) at (1.0485,2.8535) {};
    \node[above right=-0pt] at (r2) {$C \cap N$};
    \node[below left=-3pt] at (-0.25,1.0) {$C \cap N$};
    % tangency points (zeros of det M on C) with the normal field eta
    \node[tdot] (t1) at (1.3432,1.1936) {};
    \node[tdot] (t2) at (-0.5111,2.3333) {};
    \draw[->, thick, BrickRed] (t1.center) -- +(0.422,-0.353) node[right=-2pt] {$\eta$};
    \draw[->, thick, BrickRed] (t2.center) -- +(0.505,-0.217) node[below right=-3pt] {$\eta$};
  \end{tikzpicture}
  \caption{Lemma~\ref{lem:rolle-det-degree} for $n = 1$ and $s = 1$ with chain $(q_1) = (\econst^{x_1})$, so that $P_{1,1}(X, Y_1) = Y_1$. The chain equation defines a foliation of the $(x_1, x_2)$-plane by the solution curves $x_2 = c\,\econst^{x_1}$ of $\partial x_2/\partial x_1 = P_{1,1}(x_1, x_2)$ (grey), of which the graph $N = \{x_2 = q_1(x_1)\}$ is the leaf with $c = 1$. The vector $\eta = (h_{1,1}, -1) = (x_2, -1)$ is normal to the leaf through each point of the plane, and coincides with $\nabla g$, for $g = q_1(x_1) - x_2$, on $N$. The roots of $g|_C$ on the curve $C = \{g_1 = 0\}$ (here, a circle) are the points of $C \cap N$ (dots). Between consecutive roots, $C$ becomes tangent to a leaf of the foliation (diamonds) at points where $\det(M) = \det([\nabla g_1, \eta])=0$ on $C$. This locus is defined by a Pfaffian function of chain order $s - 1$, whereas $\nabla g$ involves the chain function $q_1$ itself.}
  \label{fig:rolle-det}
\end{figure}
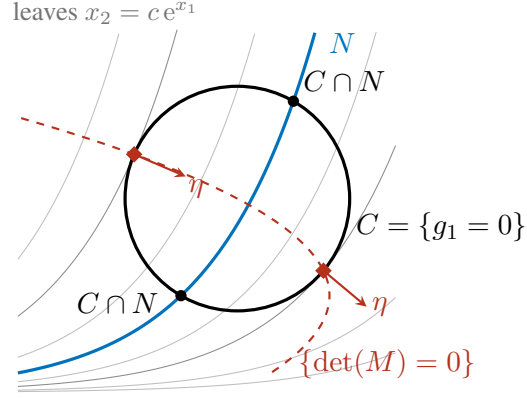

\begin{proof}[Proof of Lemma~\ref{lem:rolle-det-degree}]
The functions $g_1, \ldots, g_n$ are Pfaffian with formats $(\alpha, \beta_i, s-1)$ and share the common chain $(q_1, \ldots, q_{s-1})$. The functions $h_{s,j}$
are also Pfaffian with chain $(q_1,\dots,q_{s-1})$ by construction. Writing out $M$ explicitly:
\begin{align*}
  M(x, x_{n+1}) = \begin{bmatrix}
  \frac{\partial g_1}{\partial x_1} & \cdots & \frac{\partial g_n}{\partial x_1} & h_{s,1} \\[4pt]
  \vdots & \ddots & \vdots & \vdots \\[4pt]
  \frac{\partial g_1}{\partial x_k} & \cdots & \frac{\partial g_n}{\partial x_k} & h_{s,k} \\[4pt]
  \frac{\partial g_1}{\partial x_{k+1}} & \cdots & \frac{\partial g_n}{\partial x_{k+1}} & 0 \\[4pt]
  \vdots & \ddots & \vdots & \vdots \\[4pt]
  \frac{\partial g_1}{\partial x_n} & \cdots & \frac{\partial g_n}{\partial x_n} & 0 \\[4pt]
  \frac{\partial Q_1}{\partial y_s} & \cdots & \frac{\partial Q_n}{\partial y_s} & -1
  \end{bmatrix},
\end{align*}
where each $\frac{\partial g_i}{\partial x_j}$ (for $j \in [n]$) is
\begin{equation}\label{eq:expansion}
  \frac{\partial Q_i}{\partial x_j} + \sum_{l=1}^{s-1} \frac{\partial Q_i}{\partial y_l}\, P_{l,j}.
\end{equation}
All entries of $M$ are Pfaffian functions with chain $(q_1, \ldots, q_{s-1})$, so $\det(M)$ is as well (by Lemma~\ref{le:pfaffalgebra}).

It remains to bound the degree $\beta_{n+1}$ of $\det(M)$. We first bound the degree of $t\times t$ submatrices formed from the first $n$ rows and columns (i.e., excluding the last row and column), and then handle the full determinant via Laplace expansion. 

Let $J=\{j_1,\dots,j_t\}$ be a subset of rows and let $I=\{i_1,\dots,i_t\}$ be a subset of columns, with $t\in [n]$, and denote by $M_{J,I}$ the corresponding submatrix. Using 
the expansion~\eqref{eq:expansion}, we see that $M_{J,I}$ can 
be written as
\begin{equation*}
  M_{J,I} = A + UV^\trans,
\end{equation*}
where $A\in \R^{t\times t}$, $U\in \R^{t\times (s-1)}$, and $V\in \R^{t\times (s-1)}$ have entries given by
\begin{equation*}
  A_{ab} = \frac{\partial Q_{i_b}}{\partial x_{j_a}}, \ U_{a\ell} = P_{\ell,j_a}, \ V_{b\ell} = \frac{\partial Q_{i_b}}{\partial y_\ell} \qquad \text{for} \qquad a, b \in [t], \ell \in [s-1].
\end{equation*}
Setting $\gamma = \min \{t,s-1\}$, we can expand the determinant of $M_{J,I}$ as
\begin{equation*}
 \det(M_{J,I}) = \det(A+UV^\trans) = \det\begin{pmatrix} A & -U \\ V^{\trans} & I_{s-1}\end{pmatrix},
\end{equation*}
where $I_{s-1}$ is the $(s-1)\times (s-1)$ identity matrix. In the Leibniz expansion of this block determinant, the entries taken from the first $t$ columns (belonging to $A$ or $V^{\trans}$) contribute a total degree of at most $\sum_{i\in I}(\beta_i - 1)$, while at most $\min\{t, s-1\}$ entries of degree $\alpha$ can be taken from $-U$, since $U$ has only $t$ rows; the remaining entries come from $I_{s-1}$. Hence
\begin{equation}\label{eqn:deg-bound-submatrix}
  \deg \det(M_{J,I})\leq \sum_{i\in I} (\beta_i-1)+\min\{t,s-1\}\alpha.
\end{equation}

To bound $\deg(\det(M))$, we perform Laplace expansion along the last column $\eta$. Since $\eta$ has zero entries in rows $k+1,\dots,n$, only two types of terms contribute: the last row (entry $-1$), and rows $j\in [k]$ (entries $h_{s,j}$, each of degree~$\alpha$).

The cofactor of the last row is $\pm\det(A + UV^\trans)$ with $t = n$. Since $P_{\ell,j} = 0$ for $j > k$, the matrix $U$ has at most $k$ nonzero rows, so the block determinant expansion gives degree at most $\sum_{i=1}^n(\beta_i-1) + \min\{k,s-1\}\alpha$.

For $j\in [k]$, the cofactor of $h_{s,j}$ is $\pm\det(\widehat{M}_j)$, where $\widehat{M}_j$ is the $n\times n$ matrix obtained from the first $n$ columns of $M$ by replacing row $j$ with the last row $(\frac{\partial Q_1}{\partial y_s},\dots,\frac{\partial Q_n}{\partial y_s})$. Extending the decomposition to include~$y_s$, we write $\widehat{M}_j = \widehat{A}_j + \widehat{U}_j\widetilde{V}^\trans$, where $\widetilde{V}\in \R^{n\times s}$ extends $V$ by the column $(\frac{\partial Q_i}{\partial y_s})_{i\in [n]}$, the matrix $\widehat{A}_j$ equals $A$ with row $j$ replaced by zeros, and $\widehat{U}_j\in \R^{n\times s}$ is obtained from $U$ (extended by a zero column) by replacing row $j$ with $e_s^\trans$. In the block matrix $\bigl(\begin{smallmatrix} \widehat{A}_j & -\widehat{U}_j \\ \widetilde{V}^\trans & I_s\end{smallmatrix}\bigr)$, row $j$ reduces to $(0,\dots,0,-1)$; expanding along it and then along the resulting last row shows that $\deg\det(\widehat{M}_j) \le \sum_{i=1}^n(\beta_i - 1) + \min\{k-1,s-1\}\alpha$, since $U$ with row $j$ deleted has at most $k-1$ nonzero rows. Including the factor $h_{s,j}$ of degree $\alpha$, and using $\min\{k-1,s-1\}+1 = \min\{k,s\}$, each such term has degree at most $\sum_{i=1}^n(\beta_i-1)+\min\{k,s\}\alpha$. Since $\min\{k,s-1\} \le \min\{k,s\}$, taking the maximum over both types of terms yields
\begin{equation*}
  \beta_{n+1} = \deg(\det(M)) \le \sum_{i=1}^n (\beta_i - 1) + \min\{k, s\}\alpha,
\end{equation*}
which establishes~\eqref{eqn:rolle-det-degree}.
\end{proof}

\section{Proof of \khs{} Theorem}
\label{sec:proof}

For convenience, we restate Theorem~\ref{thm:khovanskii} here.

\getkeytheorem{thm-khovanskii}

\begin{proof}
We prove the theorem by induction on the chain order $s$. In the base case $s=0$, every $f_i$ is a polynomial of degree $\beta_i$, and the bound reduces to $\beta_1 \dotsm \beta_n$, which follows from Theorem~\ref{thm:real-bezout}. Now let $s \ge 1$ and assume the bound holds for all Pfaffian systems whose chain has order at most $s-1$.

\medskip
\noindent\textbf{Step 1: Lifting to $\boldsymbol{n+1}$ dimensions.}
Each component of $F$ has the form
\begin{equation*}
  f_i(x) = Q_i(x, q_1(x), \ldots, q_s(x)), \qquad Q_i \in \R[X_1, \ldots, X_n, Y_1, \ldots, Y_s].
\end{equation*}
Introduce a new variable $x_{n+1}$ that will play the role of $q_s$. Define $G = (g_1, \ldots, g_n)\colon \R^{n+1} \to \R^n$ and $g\colon \R^{n+1} \to \R$ by
\begin{align*}
  g_i(x, x_{n+1}) &= Q_i(x, q_1(x), \ldots, q_{s-1}(x), x_{n+1}), \qquad i \in [n], \\
  g(x, x_{n+1}) &= q_s(x) - x_{n+1}.
\end{align*}
Since $\frac{\partial g}{\partial x_{n+1}} = -1$, the value $0$ is a regular value of $g$ and $N = g^{-1}(0)$ is a smooth hypersurface in $\R^{n+1}$. By construction, for any $y \in \R^n$,
\begin{equation} \label{eqn:counting-in-n+1}
  \left\{x \in \R^n \suchthat F(x) = y\right\} = \pi_n\!\left(\left\{(x, x_{n+1}) \in N \suchthat G(x, x_{n+1}) = y\right\}\right),
\end{equation}
where $\pi_n$ denotes projection onto the first $n$ coordinates. Moreover, $(x, q_s(x))$ is a regular solution of $(G, g) = (y, 0)$ if and only if $x$ is a regular solution of $F = y$ (see Claim~\ref{claim:regular-solutions} below). Thus, it suffices to bound the number of regular solutions of $(G, g) = (0, 0)$.

\medskip
\noindent\textbf{Step 2: Applying \kh{}--Rolle.}
We first assume that $G$ is a proper map; this assumption will be removed in Claim~\ref{claim:G-proper}. Let $S$ be an arbitrary finite set of regular solutions of $(G, g) = (0, 0)$. It suffices to bound $\abs{S}$ by a quantity that depends only on $n$, $k$, and the formats of the functions involved. At each point of $S$ the differential of $(G, g)\colon \R^{n+1}\to \R^{n+1}$ is invertible, so by the inverse function theorem we may choose pairwise disjoint open neighbourhoods of the points of $S$ on which $(G, g)$ restricts to a diffeomorphism onto an open neighbourhood of $(0, 0)$. Consequently, for every $y \in \R^n$ sufficiently close to $0$, the system $(G, g) = (y, 0)$ has at least $\abs{S}$ distinct solutions. By Sard's theorem, we may in addition choose such a $y$ that is a regular value of both $G$ and of $G|_N$. Since $G$ is proper, the set $C = G^{-1}(y) \subset \R^{n+1}$ is then a smooth, compact $1$-manifold, and
\begin{equation*}
  \left\{(x, x_{n+1}) \in N \suchthat G(x, x_{n+1}) = y\right\} = \left\{(x, x_{n+1}) \in C \suchthat g(x, x_{n+1}) = 0\right\},
\end{equation*}
so that
\begin{equation} \label{eqn:S-vs-roots}
  \abs{S} \le \abs{\left\{(x, x_{n+1}) \in C \suchthat g(x, x_{n+1}) = 0\right\}}.
\end{equation}
Moreover, every root of $g|_C$ is non-degenerate. Indeed, let $z$ be such a root. Since $y$ is a regular value of $G$, we have $T_z C = \ker \diff{G}(z)$, and since $y$ is a regular value of $G|_N$, the restriction of $\diff{G}(z)$ to the $n$-dimensional subspace $T_z N$ is surjective, hence injective, so that $T_z N \cap \ker \diff{G}(z) = \{0\}$. As $T_z N = \ker \diff{g}(z)$, it follows that $\diff{(g|_C)}(z) = \diff{g}(z)|_{T_z C} \neq 0$.

Define the $(n+1) \times (n+1)$ matrices
\begin{equation} \label{eqn:matrix-J}
  J(x, x_{n+1}) := \left[\nabla g_1 \enspace \cdots \enspace \nabla g_n \enspace \nabla g \right], \qquad
  M(x, x_{n+1}) := \left[\nabla g_1 \enspace \cdots \enspace \nabla g_n \enspace \eta \right],
\end{equation}
where $\eta$ is defined in Lemma~\ref{lem:rolle-det-degree}, so that $\det(J)$ is the determinant~\eqref{eq:kho-rolle-det} associated with $G$ and $g$, and $M$ is the matrix of Lemma~\ref{lem:rolle-det-degree}. The two matrices differ only in the last column: since the chain functions depend only on $x_1, \ldots, x_k$, we have
\begin{equation*}
  \nabla g = \bigl(P_{s,1}(x_1, \ldots, x_k, q_1(x), \ldots, q_{s}(x)), \ldots, P_{s,k}(x_1, \ldots, x_k, q_1(x), \ldots, q_{s}(x)), 0, \ldots, 0, -1\bigr)^{\trans},
\end{equation*}
and $\eta$ is obtained from $\nabla g$ by substituting the free variable $x_{n+1}$ for $q_s(x)$ in the last argument of each chain polynomial. Off the hypersurface $N$, where $x_{n+1} \neq q_s(x)$, the determinants $\det(J)$ and $\det(M)$ therefore differ in general. On $N$, however, they coincide. More precisely,
\begin{equation} \label{eqn:detJ-detM}
  \det(J) = \det(M) + g \cdot E
\end{equation}
for a smooth function $E \colon \R^{n+1} \to \R$. To see this, note that the $j$-th entry of $\nabla g - \eta$ is zero for $j \notin [k]$, while for $j \in [k]$ it equals
\begin{equation*}
  P_{s,j}(x_1, \ldots, x_k, q_1(x), \ldots, q_{s-1}(x), q_s(x)) - P_{s,j}(x_1, \ldots, x_k, q_1(x), \ldots, q_{s-1}(x), x_{n+1}),
\end{equation*}
which, viewing $P_{s,j}$ as a polynomial in its last argument, factors as $(q_s(x) - x_{n+1})\, r_j(x, x_{n+1}) = g \cdot r_j$ for a smooth function $r_j$. Hence $\nabla g - \eta = g \cdot w$ with $w = (r_1, \ldots, r_k, 0, \ldots, 0)^{\trans}$, and by multilinearity of the determinant in the last column,
\begin{equation*}
  \det(J) - \det(M) = \det([\nabla g_1, \ldots, \nabla g_n, \nabla g - \eta]) = g \cdot \det([\nabla g_1, \ldots, \nabla g_n, w]),
\end{equation*}
which proves~\eqref{eqn:detJ-detM}.

By Corollary~\ref{cor:kho-rolle}, applied with $\psi = \det(M)$ (the hypothesis that $\psi$ coincides with the determinant~\eqref{eq:kho-rolle-det} on $\{g = 0\} = N$ holds by~\eqref{eqn:detJ-detM}) there exists a regular value $\delta > 0$ of $\det(M)|_C$ such that
\begin{multline} \label{eqn:new-system}
  \abs{\left\{(x, x_{n+1}) \in C \suchthat g(x, x_{n+1}) = 0\right\}} \\
  \le \abs{\left\{(x, x_{n+1}) \in C \suchthat \det(M(x, x_{n+1})) = \delta\right\}}.
\end{multline}
The right-hand side of~\eqref{eqn:new-system} is the number of solutions of the system
\begin{equation} \label{eqn:lifted-system}
  g_1(x, x_{n+1}) = y_1, \quad \ldots, \quad g_n(x, x_{n+1}) = y_n, \quad \det(M(x, x_{n+1})) = \delta,
\end{equation}
which consists of $n+1$ equations in $n+1$ unknowns, and all of these solutions are regular: at any solution $z$, the kernel of $\diff{G}(z)$ is the line $T_z C$, on which the differential of $\det(M)$ does not vanish, because $\delta$ is a regular value of $\det(M)|_C$. The key observation is that this new system is Pfaffian with chain order $s-1$, enabling the induction.

\medskip
\noindent\textbf{Step 3: Format of the new system.}
Each $g_i$ is obtained from $f_i$ by replacing $q_s(x)$ with the free variable $x_{n+1}$. Therefore, $g_1, \ldots, g_n$ are Pfaffian functions of formats $(\alpha, \beta_i, s-1)$ with respect to the common chain $(q_1, \ldots, q_{s-1})$. The last equation of~\eqref{eqn:lifted-system} involves $\det(M)$, whose last column $\eta$ belongs to the chain $(q_1, \ldots, q_{s-1})$ by construction --- in contrast to $\det(J)$, whose last column $\nabla g$ contains the chain polynomials evaluated at $q_s(x)$, and which is therefore only Pfaffian of order $s$. This is precisely why Corollary~\ref{cor:kho-rolle} was applied with the comparison function $\psi = \det(M)$ rather than with $\det(J)$ itself. By Lemma~\ref{lem:rolle-det-degree}, $\det(M)$ is Pfaffian with format $(\alpha, \beta_{n+1}, s-1)$ belonging to the chain $(q_1, \ldots, q_{s-1})$, where
\begin{equation} \label{eqn:beta-n+1-ub}
  \beta_{n+1} \le \sum_{i=1}^n (\beta_i - 1) + \min\{k, s\}\alpha.
\end{equation}

\medskip
\noindent\textbf{Step 4: Induction.}
Applying the induction hypothesis to the system~\eqref{eqn:lifted-system}, which has $n+1$ Pfaffian equations in $n+1$ unknowns with chain order $s-1$ and the chain still depending on $k$ variables, gives
\begin{align}
  \eqref{eqn:new-system} &\le 2^{\frac{(s-1)(s-2)}{2}} \left(\prod_{i=1}^{n+1} \beta_i\right) \left(\sum_{i=1}^{n+1}\beta_i - (n+1) + \min\{k,s-1\}\alpha + 1\right)^{s-1} \nonumber \\
  &= 2^{\frac{(s-1)(s-2)}{2}} \left(\prod_{i=1}^{n} \beta_i\right) \cdot \pool, \label{eqn:new-system-upto-A}
\end{align}
where
\begin{equation*}
  \pool := \beta_{n+1} \left(\sum_{i=1}^{n}\beta_i + \beta_{n+1} - n + \min\{k,s-1\}\alpha\right)^{s-1}.
\end{equation*}
Using~\eqref{eqn:beta-n+1-ub} from Lemma~\ref{lem:rolle-det-degree}, we bound
\begin{align}
  \pool &\le \beta_{n+1} \left(\sum_{i=1}^{n}\beta_i + \sum_{i=1}^n (\beta_i - 1) + \min\{k, s\}\alpha - n + \min\{k,s-1\}\alpha\right)^{s-1} \nonumber \\
  &\le \beta_{n+1} \cdot 2^{s-1}\left(\sum_{i=1}^{n}\beta_i - n + \min\{k,s\}\alpha\right)^{s-1} \nonumber \\
  &\le 2^{s-1}\left(\sum_{i=1}^{n}\beta_i - n + \min\{k,s\}\alpha\right)^{s}, \label{eqn:simplify-A-bound}
\end{align}
where in the second line we used $2\sum_{i=1}^n\beta_i - 2n + (\min\{k,s\}+\min\{k,s-1\})\alpha \le 2(\sum_{i=1}^n\beta_i - n + \min\{k,s\}\alpha)$ (since $\min\{k,s-1\}\le \min\{k,s\}$) and the bound $\beta_{n+1} \le \sum_{i=1}^n\beta_i - n + \min\{k,s\}\alpha$ in the last line.

Combining~\eqref{eqn:simplify-A-bound} with~\eqref{eqn:new-system-upto-A} yields
\begin{equation} \label{eqn:sharper-kho-proper}
  \eqref{eqn:new-system} \le 2^{\frac{s(s-1)}{2}}\left(\prod_{i=1}^n \beta_i \right)\left(\sum_{i=1}^{n}\beta_i - n + \min\{k,s\}\alpha\right)^s,
\end{equation}
which is a sharper bound than claimed (the right-hand side of~\eqref{eqn:sharper-kho-proper} omits the $+1$ in the exponent base). By~\eqref{eqn:S-vs-roots} and~\eqref{eqn:new-system}, the same quantity bounds $\abs{S}$. Since $S$ was an arbitrary finite set of regular solutions of $(G, g) = (0, 0)$, and by Step~1 these correspond bijectively to the regular solutions of $F(x) = 0$, the number of regular solutions of $F(x) = 0$ is bounded by the right-hand side of~\eqref{eqn:sharper-kho-proper}. This completes the proof under the assumption that $G$ is proper.
\end{proof}

\begin{claim} \label{claim:regular-solutions}
  A point $x \in \R^n$ is a regular solution of $F(x) = y$ if and only if $(x, q_s(x))$ is a regular solution of $(G, g)(x, x_{n+1}) = (y, 0)$.
  \end{claim}
  
\begin{proof}
By the chain rule, for each $i \in [n]$ and $j \in [n]$,
\begin{align*}
  \frac{\partial g_i}{\partial x_j}(x, x_{n+1}) &= \frac{\partial Q_i}{\partial x_j}(x, q_1(x), \ldots, q_{s-1}(x), x_{n+1}) \\
  &\qquad + \sum_{l=1}^{s-1} \frac{\partial Q_i}{\partial y_l}(x, q_1(x), \ldots, q_{s-1}(x), x_{n+1})\cdot P_{l,j}(x, q_1(x), \ldots, q_l(x)),
\end{align*}
and
\begin{equation*}
  \frac{\partial g_i}{\partial x_{n+1}}(x, x_{n+1}) = \frac{\partial Q_i}{\partial y_s}(x, q_1(x), \ldots, q_{s-1}(x), x_{n+1}).
\end{equation*}
On the hypersurface $N = \{g = 0\}$, i.e.\ at points $(x, q_s(x))$, the tangent space is
\begin{equation*}
  T_{(x, q_s(x))} N = \ker \diff{g}(x, q_s(x)) = \left\{v \in \R^{n+1} \suchthat v_{n+1} = \sum_{j=1}^n P_{s,j}(x, q_1(x), \ldots, q_s(x))\, v_j\right\}.
\end{equation*}
A direct calculation shows that for $v \in T_{(x, q_s(x))} N$,
\begin{equation*}
  \diff{g_i}(x, q_s(x))(v) = \diff{f_i}(x)(v_1, \ldots, v_n).
\end{equation*}
It follows that the $n \times n$ matrix $[\diff{f_1}(x), \ldots, \diff{f_n}(x)]$ is nonsingular if and only if $\diff{(G|_N)}(x, q_s(x))$ has rank $n$. Hence $(x, q_s(x))$ is a regular solution of $(G, g) = (y, 0)$ if and only if $x$ is a regular solution of $F = y$.
\end{proof}

\begin{claim} \label{claim:G-proper}
  The properness assumption on $G$ can be removed, and in doing so we obtain the bound stated in Theorem~\ref{thm:khovanskii}.
\end{claim}

\begin{proof}
Fix $\rho > 0$. We bound the number of regular solutions of $F(x) = 0$ contained in the open ball $B_\rho = \{x \in \R^n \suchthat \norm{x} < \rho\}$ by a quantity that does not depend on $\rho$; letting $\rho \to \infty$ then yields the same bound for the number of all regular solutions.

For $R > 0$, introduce a new variable $x_0$ and add the equation
\begin{equation*}
  f_0(x_0, x_1, \ldots, x_n) = x_0^2 + x_1^2 + \cdots + x_n^2 + q_s(x)^2 - R^2
\end{equation*}
to the system. Writing $Q_0 = X_0^2 + \sum_{i=1}^{n} X_i^2 + Y_s^2 - R^2$, we have $f_0 = Q_0(x_0, x, q_1(x), \ldots, q_s(x))$, so $f_0$ is a Pfaffian function with the same chain $(q_1, \ldots, q_s)$ and degree $\beta_0 = \deg Q_0 = 2$. The augmented system $\mathcal{S}' = \{f_0 = 0, f_1 = 0, \ldots, f_n = 0\}$ therefore consists of $n+1$ Pfaffian equations in the $n+1$ unknowns $(x_0, x_1, \ldots, x_n)$, with component-wise formats $(\alpha, \beta_i, s)$, and the chain still depends on only $k$ of the variables $\{x_0, x_1, \ldots, x_n\}$ (it is independent of $x_0$; after reordering the variables as $(x_1, \ldots, x_n, x_0)$ --- which leaves the Pfaffian structure and the formats intact, since a permutation of the variables merely permutes the chain equations~\eqref{eq:pfaffchain} --- the chain depends on the first $k$ of them, as required).

The role of the term $q_s(x)^2$ in $f_0$ is to make the lifted map proper. Indeed, the lift of Step~1, applied to $\mathcal{S}'$, replaces $q_s(x)$ by the new variable $x_{n+1}$ in each representing polynomial $Q_i$, and thus produces the map $G' = (g_0, g_1, \ldots, g_n)\colon \R^{n+2} \to \R^{n+1}$ with
\begin{equation*}
  g_0(x_0, x, x_{n+1}) = x_0^2 + x_1^2 + \cdots + x_n^2 + x_{n+1}^2 - R^2.
\end{equation*}
If $K \subset \R^{n+1}$ is compact, then $g_0$ is bounded on $(G')^{-1}(K)$, so that $(G')^{-1}(K)$ is a closed and bounded, hence compact, subset of $\R^{n+2}$; thus $G'$ is proper, and Steps~1--4 apply to the system $\mathcal{S}'$. (Note that without the term $q_s(x)^2$ the lifted coordinate $x_{n+1}$ would not be controlled by $g_0$, and $G'$ would in general fail to be proper.)

Now set $T_\rho := \max_{\norm{x} \leq \rho} \abs{q_s(x)}$ and choose $R$ with $R^2 > \rho^2 + T_\rho^2$. Every regular solution $x$ of $\{f_1 = \cdots = f_n = 0\}$ with $x \in B_\rho$ satisfies $\norm{x}^2 + q_s(x)^2 < R^2$ and therefore gives rise to exactly two solutions $(\pm x_0, x)$ of $\mathcal{S}'$, where $x_0 = \sqrt{R^2 - \norm{x}^2 - q_s(x)^2} > 0$. These solutions are regular: since $f_1, \ldots, f_n$ do not depend on $x_0$, expanding the Jacobian determinant of $(f_0, f_1, \ldots, f_n)$ at $(\pm x_0, x)$ along the column of partial derivatives with respect to $x_0$ gives $\pm 2 x_0 \det(\diff{F}(x)) \neq 0$. As distinct solutions in $B_\rho$ yield distinct pairs, the number of regular solutions of $F(x) = 0$ in $B_\rho$ is at most half the number of regular solutions of $\mathcal{S}'$.

Applying the bound~\eqref{eqn:sharper-kho-proper} (proved above under the properness assumption) to the $(n+1)$-dimensional system $\mathcal{S}'$ gives at most
\begin{equation*}
  2^{\frac{s(s-1)}{2}} \beta_0 \beta_1 \dotsm \beta_n \left(\beta_0 + \beta_1 + \dotsb + \beta_n - (n+1) + \min\{k,s\}\alpha\right)^s
\end{equation*}
regular solutions of $\mathcal{S}'$. Dividing by $2$ and setting $\beta_0 = 2$ shows that the number of regular solutions of $F(x) = 0$ in $B_\rho$ is at most
\begin{equation*}
  2^{\frac{s(s-1)}{2}} \beta_1 \dotsm \beta_n \left(\beta_1 + \dotsb + \beta_n - n + \min\{k,s\}\alpha + 1\right)^s.
\end{equation*}
This quantity depends on neither $R$ nor $\rho$. Since every regular solution of $F(x) = 0$ lies in $B_\rho$ for some $\rho > 0$, the number of regular solutions of $F(x) = 0$ is bounded by the same quantity, which is the bound stated in Theorem~\ref{thm:khovanskii}.
\end{proof}

 \section{Applications of Theorem \ref{thm:khovanskii}}
 \label{sec:applications}
 
We now prove Theorem \ref{thm:pfaffian-connected-components}, which we recall below.

\getkeytheorem{thm-pfaffian-connected-components}

The proof rests on a construction that goes back to Milnor's bound on the Betti numbers of real varieties~\cite{milnor1964betti}; see also~\cite[\S 11.5]{bochnak2013real} for the algebraic case.

 \begin{proof}
 Write $Z := \mathcal{Z}(F)$ and let $f := F_1^2 + \cdots + F_m^2$, where $F_1, \ldots, F_m$ are the components of $F$, so that $f \geq 0$ and $Z = f^{-1}(0)$. Since the $F_i$ are Pfaffian with respect to the common chain $\boldq$, it follows from Lemma~\ref{le:pfaffalgebra} and Remark~\ref{re:1} that $f \in \Pfaff_{\alpha, 2\beta, s}(\R^n)$, again with chain $\boldq$. For $R > 0$ and $\eps > 0$, define
 \begin{equation*}
   F_{\eps}(x) := f(x) + \eps\left(\norm{x}^2 - R^2\right).
 \end{equation*}
 The added term is a polynomial of degree $2 \leq 2\beta$, so $F_{\eps}$ is again Pfaffian with chain $\boldq$ and format $(\alpha, 2\beta, s)$; by Lemma~\ref{le:pfaffalgebra}(3) and Remark~\ref{re:1}, its partial derivatives belong to $\Pfaff_{\alpha, \alpha+2\beta-1, s}(\R^n)$, still with chain $\boldq$. For $c = (c_2, \ldots, c_n) \in \R^{n-1}$, consider the system of $n$ equations in $n$ unknowns
 \begin{equation}\label{eqn:crit-system}
   F_{\eps}(x) = 0, \qquad
   \frac{\partial F_{\eps}}{\partial x_j}(x) - c_j\,\frac{\partial F_{\eps}}{\partial x_1}(x) = 0,
   \quad j = 2, \ldots, n.
 \end{equation}
 The functions appearing in~\eqref{eqn:crit-system} are Pfaffian with common chain $\boldq$ and degrees $\beta_1 = 2\beta$ and $\beta_j = \alpha + 2\beta - 1$ for $j \geq 2$, and the chain depends on $k$ of the variables. By Theorem~\ref{thm:khovanskii}, the number of regular solutions of~\eqref{eqn:crit-system} is therefore at most
 \begin{equation*}
   2^{\frac{s(s-1)}{2}}\, 2\beta\, (\alpha + 2\beta - 1)^{n-1}
   \left(2\beta + (n-1)(\alpha + 2\beta - 1) - n + \min\{k, s\}\alpha + 1\right)^{s},
 \end{equation*}
 which coincides with~\eqref{eq:jones}, since $2\beta + (n-1)(\alpha+2\beta-1) - n + 1 = n(\alpha + 2\beta - 2) - \alpha + 2$. It thus suffices to show that for any $N$ distinct connected components $C_1, \ldots, C_N$ of $Z$, there are choices of $R, \eps$ and $c$ for which the system~\eqref{eqn:crit-system} has at least $N$ regular solutions.

 \medskip
 \noindent\textbf{Step 1: Separating the components inside a compact set.}
 Fix points $p_i \in C_i$ for $i \in [N]$ and choose $R > \max_i \norm{p_i}$, and let $\bar{B}_R := \{x \in \R^n \suchthat \norm{x} \leq R\}$. Consider the sets
 \begin{equation*}
  D_{\eps} := \left\{x \in \R^n \suchthat F_{\eps}(x) \leq 0\right\}
   = \left\{x \in \R^n \suchthat f(x) \leq \eps\left(R^2 - \norm{x}^2\right)\right\}.
 \end{equation*}
 Each $D_\eps$ is closed and, since $f \geq 0$ forces $F_\eps > 0$ outside $\bar{B}_R$, contained in $\bar{B}_R$; hence $D_\eps$ is compact. The sets $D_\eps$ decrease as $\eps$ decreases, with $\bigcap_{l \in \N} D_{1/l} = Z \cap \bar{B}_R$: if $x \in D_{1/l}$ for all $l$, then $f(x) \leq \frac{1}{l}(R^2 - \norm{x}^2) \to 0$, so $f(x) = 0$; the converse inclusion is clear. Let $K_i^{\eps}$ denote the connected component of $D_\eps$ containing $p_i$ (note that $F_\eps(p_i) = \eps(\norm{p_i}^2 - R^2) < 0$). For $\eps' < \eps$ we have $K_i^{\eps'} \subseteq K_i^{\eps}$, and
 \begin{equation*}
   A_i := \bigcap_{l \in \N} K_i^{1/l}
\end{equation*}
 is the intersection of a decreasing sequence of compact connected sets, and is therefore compact and connected\footnote{If $A_i = A \sqcup B$ with $A, B$ non-empty, compact and disjoint, choose disjoint open sets $U \supset A$ and $V \supset B$. The sets $K_i^{1/l} \setminus (U \cup V)$ are compact, nested, and non-empty, since a connected set meeting both $U$ and $V$ cannot be contained in $U \cup V$; a point in their intersection would lie in $A_i \setminus (U \cup V) = \emptyset$, a contradiction.}. Since $A_i \subseteq Z \cap \bar{B}_R$ is connected and contains $p_i$, we conclude that $A_i \subseteq C_i$. In particular, $p_j \notin A_i$ for $j \neq i$, so there exists $l_{ij} \in \N$ with $p_j \notin K_i^{1/l}$ for all $l \geq l_{ij}$. Setting $\eps_0 := 1/\max_{i \neq j} l_{ij}$ (with $\eps_0 := 1$ if $N = 1$), the components $K_1^{\eps}, \ldots, K_N^{\eps}$ of $D_\eps$ are pairwise distinct, and hence pairwise disjoint, for every $\eps < \eps_0$.

 \medskip
 \noindent\textbf{Step 2: A compact smooth hypersurface enclosing the components.}
 We next choose $\eps < \eps_0$ such that $0$ is a regular value of $F_\eps$, so that $X_\eps := F_\eps^{-1}(0)$ is a smooth compact hypersurface. On the open ball $B_R = \{\norm{x} < R\}$, define
 \begin{equation*}
   G := \frac{f}{R^2 - \norm{x}^2}.
 \end{equation*}
 For $x \in B_R$ we have $F_\eps(x) = 0$ if and only if $G(x) = \eps$, and at any such point a direct computation gives $\nabla F_\eps(x) = (R^2 - \norm{x}^2)\,\nabla G(x)$. By Sard's theorem, we may choose $\eps \in (0, \eps_0)$ to be a regular value of $G$; then $\nabla F_\eps$ does not vanish at any point of $X_\eps \cap B_R$. At a point $x \in X_\eps$ with $\norm{x} = R$, we have $f(x) = F_\eps(x) = 0$, so $x$ is a global minimum of $f$ and $\nabla f(x) = 0$, whence $\nabla F_\eps(x) = 2\eps x \neq 0$. Since $X_\eps \subseteq D_\eps \subseteq \bar{B}_R$, this shows that $0$ is a regular value of $F_\eps$ and that $X_\eps$ is compact.

 Write $K_i := K_i^{\eps}$. The topological boundary of $K_i$ satisfies $\emptyset \neq \partial K_i \subseteq X_\eps$. Indeed, if $x \in \partial K_i$ had $F_\eps(x) < 0$, then a connected open neighbourhood of $x$ would be contained in $D_\eps$ and meet $K_i$, hence be contained in $K_i$, making $x$ an interior point; and if $\partial K_i$ were empty, the non-empty set $K_i$ would be open and closed in $\R^n$, contradicting compactness. Now let $\Gamma_i$ be a connected component of $X_\eps$ containing a point of $\partial K_i$. Since $\Gamma_i \subseteq X_\eps \subseteq D_\eps$ and $\Gamma_i \cap K_i \neq \emptyset$, the union $\Gamma_i \cup K_i$ is a connected subset of $D_\eps$ containing $p_i$, so $\Gamma_i \subseteq K_i$. As the sets $K_1, \ldots, K_N$ are pairwise disjoint, the components $\Gamma_1, \ldots, \Gamma_N$ of $X_\eps$ are pairwise distinct.

 \medskip
 \noindent\textbf{Step 3: All solutions of the system are regular.}
 Assume $n \geq 2$; for $n = 1$, the tuple $c$ is empty, the system~\eqref{eqn:crit-system} reduces to the single equation $F_\eps = 0$, and every point of the finite set $X_\eps$ is a regular solution, since $0$ is a regular value of $F_\eps$. Set $u := (1, c_2, \ldots, c_n) \in \R^n$ and $\ell(x) := \ip{u}{x}$. We first observe that the solutions of~\eqref{eqn:crit-system} are exactly the critical points of $\ell|_{X_\eps}$. At a solution $x$, we must have $\partial F_\eps/\partial x_1(x) \neq 0$, since otherwise all partial derivatives, and hence $\nabla F_\eps(x)$, would vanish, contradicting the regularity of the value $0$; therefore $\nabla F_\eps(x) = \frac{\partial F_\eps}{\partial x_1}(x)\, u$, so that $u$ is orthogonal to $T_x X_\eps = \ker \diff{F_\eps}(x)$, which is precisely the condition for $x$ to be a critical point of $\ell|_{X_\eps}$. Conversely, at a critical point $x$ of $\ell|_{X_\eps}$, the gradient $\nabla F_\eps(x)$ is a non-zero multiple $\lambda u$ of $u$, and comparing first coordinates gives $\lambda = \partial F_\eps / \partial x_1 (x)$, so $x$ solves~\eqref{eqn:crit-system}.

 On the open set $V := \{x \in \R^n \suchthat \partial F_\eps/\partial x_1(x) \neq 0\}$, define the smooth map
 \begin{equation*}
   \psi := \left(\frac{\partial F_\eps/\partial x_2}{\partial F_\eps/\partial x_1}, \ldots,
   \frac{\partial F_\eps/\partial x_n}{\partial F_\eps/\partial x_1}\right)\colon V \to \R^{n-1}.
 \end{equation*}
 By the preceding paragraph, the solution set of~\eqref{eqn:crit-system} is $\{x \in X_\eps \cap V \suchthat \psi(x) = c\}$. By Sard's theorem, applied to the restriction of $\psi$ to the $(n-1)$-dimensional manifold $X_\eps \cap V$, we may choose $c$ to be a regular value of $\psi|_{X_\eps \cap V}$. We claim that then every solution $x^*$ of~\eqref{eqn:crit-system} is regular. On $V$, the functions of the system satisfy $\partial F_\eps/\partial x_j - c_j\, \partial F_\eps /\partial x_1 = \frac{\partial F_\eps}{\partial x_1} \cdot (\psi_j - c_j)$, and since $\psi_j(x^*) = c_j$, taking gradients gives
 \begin{equation*}
   \nabla \left(\frac{\partial F_\eps}{\partial x_j} - c_j \frac{\partial F_\eps}{\partial x_1}\right)(x^*)
   = \frac{\partial F_\eps}{\partial x_1}(x^*)\, \nabla \psi_j(x^*).
 \end{equation*}
 The Jacobian determinant of the system~\eqref{eqn:crit-system} at $x^*$ therefore equals
 \[
 \left(\frac{\partial F_\eps}{\partial x_1}(x^*)\right)^{n-1} \det([\nabla F_\eps, \nabla \psi_2, \ldots, \nabla \psi_n]^{\trans}).
 \]
 The latter determinant is non-zero if and only if no non-zero vector lies in $\ker \diff{F_\eps}(x^*) \cap \ker \diff{\psi}(x^*)$; as $\ker \diff{F_\eps}(x^*) = T_{x^*} X_\eps$ has dimension $n-1$, this holds if and only if $\diff{(\psi|_{X_\eps \cap V})}(x^*)$ is injective, or equivalently surjective, which is guaranteed by the choice of $c$.

 \medskip
 \noindent\textbf{Step 4: Counting.}
 For each $i \in [N]$, the restriction of $\ell$ to the compact non-empty set $\Gamma_i$ attains a minimum at some point $z_i$. Since $\Gamma_i$, being a connected component of $X_\eps$, is open in $X_\eps$, the point $z_i$ is a local minimum, and hence a critical point, of $\ell|_{X_\eps}$; by Step~3, it is a regular solution of~\eqref{eqn:crit-system}. (For $n = 1$, take for $z_i$ any point of $\Gamma_i$.) Since the $\Gamma_i$ are pairwise disjoint, the solutions $z_1, \ldots, z_N$ are distinct, so the system~\eqref{eqn:crit-system} has at least $N$ regular solutions. As shown above, it follows that $N$ is bounded by~\eqref{eq:jones}, and since $N$ was an arbitrary finite number of distinct connected components, the claim follows.
 \end{proof}

\printbibliography

\end{document}